\documentclass[11pt, reqno]{amsart}
\usepackage{amsmath}
\usepackage{amsthm}
\usepackage{amssymb}
\usepackage{amscd}
\usepackage{amsbsy}
\usepackage{epsfig}
\usepackage{graphicx}
\usepackage{psfrag}
\usepackage{hyperref}
\theoremstyle{plain}
\newtheorem{theorem}{Theorem}[section]
\newtheorem{lemma}[theorem]{Lemma}

\newtheorem{proposition}[theorem]{Proposition}
\newtheorem{corollary}[theorem]{Corollary}

\theoremstyle{definition}

\theoremstyle{remark}
\newtheorem{remark}[theorem]{Remark}

\def\R{\ensuremath{\mathbb R}}
\def\N{\ensuremath{\mathbb N}}
\def\NN{\ensuremath{\mathbb N}}

\def\PP{\ensuremath{\mathbb P}}
\def\EE{\ensuremath{\mathbb E}}
\def\EEE{\ensuremath{\mathcal E}}

\def\C{\ensuremath{\mathcal C}}
\def\L{\ensuremath{\mathcal L}}

\def\p{\ensuremath{\mathbb P}}
\def\F{\ensuremath{\mathcal F}}

\def\es{{\emptyset}}
\def\sm{\setminus}
\def\eps{\varepsilon}

\DeclareMathOperator{\var}{var}
\DeclareMathOperator{\esss}{ess-sup}

\numberwithin{equation}{section}

\begin{document}
\author[J. Rousseau]{J\'er\^ome Rousseau}
\address{J\'er\^ome Rousseau\\
Departamento de Matem\'atica\\
Universidade Federal da Bahia\\
Avenida Ademar de Barros s/n\\
40170-110 Salvador, BA\\
Brasil} \email{\href{mailto:jerome.rousseau@ufba.br}{jerome.rousseau@ufba.br}}
\urladdr{\url{http://www.sd.mat.ufba.br/~jerome.rousseau/}}

\author[M. Todd]{Mike Todd}
\address{Mike Todd\\ Mathematical Institute\\
University of St Andrews\\
North Haugh\\
St Andrews\\
KY16 9SS\\
Scotland} \email{\href{mailto:m.todd@st-andrews.ac.uk}{m.todd@st-andrews.ac.uk}}
\urladdr{\url{http://www.mcs.st-and.ac.uk/~miket/}}

\thanks{This work was partially supported by the EU-funded Brazilian-European Partnership on Dynamical Systems (BREUDS), FAPESB and CNPQ}

\title{Hitting times and periodicity in random dynamics}

\maketitle

\begin{abstract}
We prove quenched laws of hitting time statistics for random subshifts of finite type. In particular we prove a dichotomy between the law for periodic and for non-periodic points. We show that this applies to random Gibbs measures.
\end{abstract}

\section{Introduction}

The study of hitting time statistics (HTS) for a dynamical system is part of the broader study of recurrence properties in dynamical systems.  The basic idea is that for a dynamical system $f:X\to X$ equipped with some invariant probability measure $\mu$, given a sequence of shrinking sets $(A_n)_n\subset X$, one can look at the distribution of the random variable 
$$\tau_{A_n}(x):=\inf\{k\ge 1:f^k(x)\in A_n\},$$
as $n\to \infty$.  When the system is well-behaved and the sets $(A_n)_n$ are well-chosen, for example if $(X, f)$ is a subshift of finite type, $\mu$ a Bernoulli measure and each $A_n$ is an $n$-cylinder around some $\mu$-typical point $z$, the suitably normalised random variable $\tau$ is exponential, i.e., for each $t>0$,
\begin{equation}
\lim_{n\to \infty}\mu\left(\tau_{A_n}(\cdot)>\frac t{\mu(A_n)} \right)=e^{- t}.
\label{eq:HTS}
\end{equation}
In fact this result holds for many non-uniformly hyperbolic systems and for $(A_n)_n$ chosen to be balls around a typical point $z$ also: see the reviews \cite{Coe00, Sau09, Hay13}.  Remarkably it is not necessary that the system be exponentially mixing (e.g. \cite{BruSauTroVai03, BruVai03, BruTod09, Sau09}), nor indeed for there be any mixing information at all \cite{BruTod09}. 
 Note that sets $(A_n)_n$ which are not simply balls/cylinders shrinking down to some point can lead to very different phenomena \cite{KupLac05, Lac02, DowLac11}. We also remark on the generalisation of these ideas to the related `observational' viewpoint in \cite{Rou14}.

The next question that arises here is: what about non-typical points?  It has been shown in \cite{Hir, FerPol12, FreFreTod12} that if our point $z$ is periodic, then the distribution 
in \eqref{eq:HTS} is of the form $e^{-\Theta t}$ where $\Theta\in(0,1)$ is a parameter which takes into account the amount of repulsion at $z$.  Indeed, in a general subshift of finite type setting in \cite{FerPol12}, and a more restricted setting in \cite{FreFreTod12},  a dichotomy was proved: it was shown that the limit in \eqref{eq:HTS} exists for any $z$,  and is  $e^{-\Theta t}$ for $\Theta\in (0,1)$ in the case that $z$ is periodic, and $\Theta=1$ in the case when $z$ is non-periodic.  Some of these results were motivated by the connection of HTS laws to Extreme Value Laws (EVL), see \cite{Col01, FreFreTod10}, one reason why $\Theta$ can be referred to as the \emph{extremal index}.  

A natural direction to expand this theory is to the realm of random dynamical systems (see \cite{RouSau, MarieRou} for the first study of return times for random systems), so often we think of dynamical systems $\{f_{\omega_0}:X\to X\}$, where $\omega_0$ is chosen randomly from some set. Thus the $n$-th iterate is of the form $f_{\omega_n}\circ f_{\omega_{n-1}}\circ \cdots \circ f_{\omega_1}$.
The randomness can come about in a variety of ways.  The approach of small random perturbations, e.g. for $f_{\omega_0}=f+\omega_0$ for some fixed $f$, was taken in \cite{AytFreVai14}, in which case they think of the randomness as noise.
  Then the authors were able to derive the HTS laws after integrating over all noise: the annealed approach.  There they obtained exponential HTS, using an approach coming from extreme value theory. The same type of results were obtained in \cite{Rou14}, using observations for more general random dynamical systems: for a family of maps $\{f_\omega\}_\omega$, the randomness came from a dynamical system $(\Omega, \theta, \PP)$ and the random orbit is given by $$f_\omega^n(x)=f_{\theta^{n}\omega}\circ f_{\theta^{n-1}\omega}\circ\dots\circ f_\omega(x)$$ with a weaker mixing assumption. It is worth noting that in both papers the law is given with respect to the invariant measure of the associated skew-product.

The quenched approach to this problem is to take a random realisation $\omega\in \Omega$, consider sample-stationary measures $\mu_\omega$ (these measures which will be defined properly in Section~\ref{sec-result}, can be characterized as measures satisfying $({f_\omega})_*\mu_\omega=\mu_{\theta\omega}$ for a.e. $\omega$) and derive results with respect to them.  
In \cite{RouSauVar13}, the authors showed that for some symbolic random dynamical systems with sufficiently good mixing properties, for typical realisation $\omega$ and for typical $z$, 
$$\lim_{n\to \infty}\mu_\omega\left(\tau_{C_n(z)}(\cdot)>\frac t{\mu(C_n(z))} \right)=e^{-t},$$
where $C_n(z)$ denotes the random $n$-cylinder around $z$.  The most complete results were proved for random subshifts of finite type on a finite alphabet.  We note that they could also integrate over the randomness to achieve an annealed law.

In this paper, we synthesise the random approach of  \cite{RouSauVar13} with the approach to the extremal index in \cite{FreFreTod12}, so that while we assume typical realisations $\omega$, we are considering particular periodic points $z$.  We prove that we still obtain a non-trivial extremal index, even in this random case.  Indeed we prove a dichotomy: either the random HTS law is $e^{-\Theta t}$ for $\Theta\in (0, 1)$ and the point $z$ is periodic, or the HTS law is $e^{-t}$.  We apply this to some random Gibbs measures, including the countable alphabet case, which we prove satisfy the mixing properties we require.

Note that the appearance of periodic points in our random shifts is very natural.  In the most elementary example, all the shifts are the same subshift of finite type and the randomness comes from the measures on them (e.g. the full shift on two symbols with $(\omega, 1-\omega)$-Bernoulli measure), so periodic points appear as usual.  Even in much more complicated cases, periodic strings of symbols still appear: and when they don't, our dichotomy implies that $\Theta =1$.

We prove the dichotomy of the HTS law for random subshifts of finite type in Section \ref{sec-proof}, while the settings is described and the principal results are enunciated in Section \ref{sec-result}. To ease the reading of the proof in the random setting, the proof of the HTS law for periodic point in the deterministic case is given in Section \ref{sec-det}. In Section \ref{sec:Theta ex}, the problem of the existence of the extremal index is adressed. Finally, we apply our results to a family of random subshifts with random Gibbs measures in Section \ref{sec:examples}.

\section{Hitting time statistics dichotomy for random subshifts}\label{sec-result}

Let $(\Omega,\theta,\PP)$ be an invertible ergodic measure preserving system, set $X=\NN^{\NN_0}$ and let
$\sigma: X \to X$ denote the shift. Let $A=\left\{A(\omega)=(a_{ij}(\omega)):\omega\in \Omega\right\}$ be a random transition matrix, i.e., for any $\omega\in\Omega$, $A(\omega)$ is an $\N\times \N$-matrix with entries in $\{0,1\}$, at least one non-zero entry in each row and each column and such that $\omega\mapsto a_{ij}(\omega)$ is measurable for any $i\in\NN$ and $j\in\NN$.  
For any $\omega\in \Omega$
define
\begin{equation*}
X_\omega =\{x=(x_0,x_1,\ldots)\colon x_i\in \N \text{ and } a_{x_i x_{i+1}}(\theta^i\omega)=1\text{ for all } i\in\NN\}
\end{equation*}
and
\begin{equation*}
\EEE = \{(\omega,x)\colon \omega\in\Omega,x\in X_\omega\} \subset \Omega\times X.
\end{equation*}
We consider the random dynamical system coded by the skew-product $S : \EEE \to \EEE$ given by
$S(\omega,x)= (\theta \omega,\sigma x)$.  
While we allow infinite alphabets here, we nevertheless call $S$ a \emph{random subshift of finite type} (SFT).  Assume that $\nu$ is an $S$-invariant probability
measure with marginal $\PP$ on $\Omega$.  Then we let $(\mu_\omega)_\omega$ denote
its decomposition on $X_\omega$, that is, $d\nu(\omega,x)=d\mu_\omega(x)d\PP(\omega)$. 
The measures $\mu_\omega$ are called the \emph{sample measures}.  Note $\mu_\omega(A)=0$ if $A\cap X_\omega=\es$. We denote by $\mu=\int \mu_\omega \, d\PP$ the marginal of $\nu$ on $X$.

For $y\in X$ we denote by $C_n(y)=\{z \in X : y_i=z_i \text{ for all } 
0\le i\le n-1\} $ the  \emph{$n$-cylinder} that contains $y$. Let $\F_0^n$ be the sigma-algebra in $X$ 
generated by all the $n$-cylinders.

We assume the following: there are constants $h_0>0$, $c_0>0$ and a summable function $\psi$ such that
for all $m,n,g\in \N$, $A\in\F_0^n$ and $B\in \F_0^{m}$: 

\begin{itemize}
\item[(I)]  the marginal measure $\mu$ satisfies
\[
\left|\mu(A\cap\sigma^{-g-n}B) -\mu(A)\mu(B)\right|\le \psi(g);
\]
\item[(II)] for $\PP$-almost every $\omega\in\Omega$, if $y\in X_\omega$ and $n\ge 1$ then  
  $c_0^{-1} e^{-h_0 n}\leq \mu(C_n(y))$; 
\item[(III)] for $\PP$-almost every $\omega\in\Omega$,
\[
\left|\mu_\omega(A\cap\sigma^{-g-n}B) -\mu_\omega(A)\mu_{\theta^{n+g}\omega}(B)\right|\le \psi(g)\mu_\omega(A)\mu_{\theta^{n+g}\omega}(B);
\]

\item[(IV)] the sample measures satisfy
\[\underset{\omega\in\Omega}\esss\sup_{x\in X}\mu_\omega(C_1(x))<1.\]
\end{itemize}

First of all, we give a result on the measure of cylinders, following \cite{GS}, that will be used several times in our paper.
\begin{lemma} For a random SFT such that assumptions (III) and (IV) hold, there exist $c_1, c_2>0$ and $h_1>0$ such that for any $y\in X$, $n\ge 1$ and $m\ge 1$, for $\PP$-almost every $\omega\in\Omega$
\begin{equation}\label{measurecyl}
 \mu_\omega(C_n(y))\le c_1 e^{-h_1 n} 
 \end{equation}
and
\begin{equation}\label{measuresumcyl}
\sum_{k=m}^{n}\mu_\omega(C_n(y)\cap \sigma^{-k}C_n(y))\le c_2 e^{-h_1 m}\mu_\omega(C_n(y)). \end{equation}
\end{lemma}
\begin{proof} A straightforward adaptation to the random setting of the proof of Lemma 1 in \cite{GS} gives the result. \end{proof}

Our first result works only in the case of finite alphabet since assumption (II) cannot be fulfilled otherwise. A stronger mixing assumption for the marginal measure $\mu$ will allow us to treat the case of infinite alphabets.

\begin{theorem}\label{thprin}
Assume  (I)--(IV) hold and that there exists a constant $q>2\frac{h_0}{h_1}$ such that $\psi$ satisfies $\psi(g)g^q\to0$ as $g\to+\infty$. Let $z\in X$.  Then for $\PP$-almost every $\omega$, 
 either
\begin{itemize}
\item[(a)] $z$ is a periodic point of period $p$ and if the limit $\Theta:=\lim_{n\to \infty}\frac{\mu\left(C_n(z)\sm C_{n+p}(z)\right)}{\mu(C_n(z))}$ exists, then for all $t\ge 0$ we have
$$\lim_{n\to \infty}\mu_\omega\left( \tau_{C_n(z)}(\cdot)>\frac t{\mu(C_n(z))} \right)=e^{-\Theta t};$$
or
\item[(b)]
for all $t\ge 0$ we have
$$\lim_{n\to \infty}\mu_\omega\left( \tau_{C_n(z)}(\cdot)>\frac t{\mu(C_n(z))} \right)=e^{-t}.$$
\end{itemize}
\label{thm:main1}
\end{theorem}

This dichotomy can be compared with both the dichotomy for deterministic systems, see for example \cite{Ab} or \cite[Theorem A]{AytFreVai14}, and to the typical case for sample measures \cite[Theorem 1]{RouSauVar13}. Analogously to Corollary 2 of the latter paper we easily obtain the following annealed law.

\begin{corollary}\label{coroprin}
Assume (I)--(IV) hold and that there exists a constant $q>2\frac{h_0}{h_1}$ such that $\psi$ satisfies $\psi(g)g^q\to0$ as $g\to+\infty$.  Then for $z\in X$, either 
\begin{itemize}
\item[(a)] $z$ is a periodic point of period $p$ and if the limit $\Theta:=\lim_{n\to \infty}\frac{\mu\left(C_n(z)\sm C_{n+p}(z)\right)}{\mu(C_n(z))}$ exists, then for all $t\ge 0$ we have
$$\lim_{n\to \infty}\mu\left( \tau_{C_n(z)}(\cdot)>\frac t{\mu(C_n(z))} \right)=e^{-\Theta t};$$
or
\item[(b)]
for all $t\ge 0$ we have
$$\lim_{n\to \infty}\mu\left( \tau_{C_n(z)}(\cdot)>\frac t{\mu(C_n(z))} \right)=e^{-t}.$$
\end{itemize}
\label{cor:main1}
\end{corollary}

\begin{remark}
\begin{enumerate}
\item 
In both of the above results, for the standard examples the limit $\Theta$ does indeed exist, see Section~\ref{sec:Theta ex}.
\item In line with \cite{AytFreVai14}, and in contrast to \cite{RouSauVar13}, in (III) we require quite strong mixing properties.  As we demonstrate in Section~\ref{sec:examples}, these are satisfied by standard examples.

\item Our results are stated for almost all $\omega$.  We note that firstly $\mu_\omega$ is only assured to exist for a full measure set of $\omega$, so outside this set we cannot claim anything.  However, consider \cite[Examples 19, 20]{RouSauVar13} where $\Omega=X=\{0,1\}^{\N_0}$ with $\mu_{\omega}[x_0, \ldots, x_n]=p_{x_0}(\omega)p_{x_1}(\theta\omega)\cdots p_{x_n}(\theta^n\omega)$ where 
\begin{equation*}
p_{x_i}(\omega)=\begin{cases} p(\omega) & \text{ if }  x_i=0,\\
1-p(\omega) & \text{ if }  x_i=1,
\end{cases}
\end{equation*}
and 
\begin{equation*}
p(\omega)=\begin{cases} \alpha & \text{ if }  \omega_0=0,\\
\beta & \text{ if }  \omega_0=1,
\end{cases}
\end{equation*}
for $\alpha, \beta\in (0,1)$.  Then for any periodic point $z$, one can choose non-typical $\omega$ so that  
the relevant $\Theta$ in the law $e^{-\Theta t}$ exists, but is not $$\lim_{n\to \infty}\frac{\mu\left(C_n(z)\sm C_{n+p}(z)\right)}{\mu(C_n(z))},$$ namely by choosing $0$s in $\omega$ with a non-typical frequency.  Also we can choose $\omega$ such that there is no limit $\Theta$ by choosing $\omega$ with suitably long strings of 0s followed by much longer strings of 0s etc., which also gives non-existence of a law for typical $z$.  Note that this is principally due to the failure of Lemma~\ref{lem:MN} for non-typical $\omega$.

\item 
Many of the ideas here work beyond the random SFT setting, but we will restrict ourselves to the SFT case, both for expository purposes and to give a clear connection to straightforward applications.

\end{enumerate}
\end{remark}

Allowing infinite alphabets, one cannot hope to verify assumption (II) but the result in Theorem~\ref{thprin} is still satisfied replacing assumption (I) by assumption (I'):
\begin{itemize}
\item[(I')]  the marginal measure $\mu$ satisfies
\[
\left|\mu(A\cap\sigma^{-g-n}B) -\mu(A)\mu(B)\right|\le  \psi(g)\mu(A)\mu(B).
\] 
\end{itemize}

\begin{theorem}\label{theoinfalph}
If (I'), (III) and (IV) hold and there exists a constant $q>2\frac{h_0}{h_1}$ such that $\psi$ satisfies $\psi(g)g^q\to0$ as $g\to+\infty$, then the conclusions of Theorem~\ref{thprin} and Corollary~\ref{coroprin} are satisfied.
\end{theorem}

\section{Hitting time statistics of periodic points for deterministic dynamical systems}\label{sec-det}

Before proving Theorem~\ref{thm:main1}, we will provide a proof in the deterministic case for periodic points.  While the result is already known by \cite{FreFreTod12}, the proof uses the framework of \cite{hsv} (we will use the presentation of that in \cite{Sau09}) which is the way we will finally prove our main theorem, so the following section can be seen as a warm-up.

Suppose that $X$ carries a Riemannian metric.  Let $f:X\to X$ be a measurable map, locally differentiable at some periodic point $z\in X$ with an ergodic invariant probability measure $\mu$.

We say that we have decay of correlations against $L^1$ if the following holds.
Let $C$ be some Banach space, $\varphi\in \C$ and $\psi\in L^1(\mu)$.  We require
$$\frac1{\|\varphi\|_{\C}\|\psi\|_{L^1}}\left|\int\varphi(\psi\circ f^n)~d\mu-\int\varphi~d\mu\int\psi~d\mu\right|\to 0$$
as $n\to \infty$.
As in, for example \cite{AytFreVai14}, we will assume decay of correlations against $L^1$ with the space $\C$ containing characteristic functions on balls and having a constant $K$ so that $\|1_A\|_{\C}<K$ for all balls $A$.

\begin{theorem}
Suppose that $(X, f,\mu)$ has decay of correlations against $L^1$ of at least rate $n^{-(1+\eps)}$, for some $\eps>0$.  Then for $z$ a $p$-periodic point, if the limit $\Theta:=1-\lim_{r\to 0} \frac{\mu(f^{-p}(B_r(z))\cap B_r(z))}{\mu(B_r(z))}$ exists then
$$\lim_{r\to 0}\mu\left(\tau_{B_r(z)}(\cdot)>\frac t{\mu(B_r(z))} \right)=e^{-\Theta t}.$$
\label{thm:main det}
\end{theorem}

This can be interpreted in the same way as \cite[Proposition 2]{FreFreTod12} (see also \cite[Section 3]{AytFreVai14}), but we present a proof here to clarify our use of the approach described in \cite{Sau09}.

We will now prove an analogue of \cite[Lemma 41]{Sau09}.
Given sets $A, A'\subset X$, where $\mu(A)>0$, define 
$$\delta(A, A'):=\sup_{j\geq p}\left|\frac{\mu(A')}{\mu(A)}\mu(\tau_A>j)-\mu_A(\tau_A>j)\right|.$$
In applications, $A'\subset A$, so $\frac{\mu(A')}{\mu(A)}\le 1$.  In fact we will take sequences $(A_n, A_n')_n$ with measures going to zero where $A'$ will be called the \emph{escaping set} of $A$ which will mean that for any fixed $p\in N$, there is $n(p)$ such that $n\ge n(p)$ implies that $\tau_{A_n}>p$.

\begin{lemma}
For any $n>p$,
$$\left|\mu(\tau_A>n)-(1-\mu(A'))^{n-p}\mu(\tau_A>p)\right|\le \frac{\mu(A)}{\mu(A')}\delta(A, A').$$
\label{lem:exp approx}
\end{lemma}

\begin{proof}
We will use an inductive argument starting, for $k\ge  p$, with the estimate
\begin{align*}
\left|\mu(\tau_A>k+1)-(1-\mu(A'))\mu(\tau_A>k)\right|& =\left| \mu(A')\mu(\tau_A>k)-\mu(\tau_A=k+1)\right|\\
& =\left|\mu(A')\mu(\tau_A>k)-\mu(A\cap\{\tau_A>k\})\right|\\
&\le \mu(A)\delta(A, A'),
\end{align*}
where the second equality follows from (1) in \cite{Sau09}.
Then, for $n>p$,
\begin{align*}
\left|\mu(\tau_A>n)-(1-\mu(A'))^{n-p}\mu(\tau_A>p)\right|& \le \sum_{k=0}^{n-p-1}(1-\mu(A'))^k \mu(A)\delta(A, A')\\
&\leq\frac{\mu(A)}{\mu(A')}\delta(A,A'),
\end{align*}
as required.
\end{proof}

\begin{lemma}\label{lem:delta0}
For $z$ a $p$-periodic point, $A=B_r(z)$  and $A'=B_r(z)\setminus f^{-p}B_r(z)$, we have
$$\delta(A, A')\underset{r\rightarrow0}\rightarrow 0.$$
\end{lemma}

\begin{proof} Let the correlations decay rate be $n^{-\beta}$: we will insert $\beta=1+\eps$ when we need to make the final estimates.
We follow the argument of \cite[Theorem 40]{Sau09}, which essentially means providing an analogue of Lemma 45 of that paper.  We assume that $r$ is so small that $B_r(z), f(B_r(z)),\ldots, f^{p-1}(B_r(z))$ are disjoint. We write $E_n:=\{\tau_A> n\}$ and assume $n\geq p$.  Then we are interested in 
\begin{align*}
\left|\mu_A(E_n)-\frac{\mu(A')}{\mu(A)}\mu(E_n)\right|&=\frac1{\mu(A)}\left|\mu(A\cap E_n)-\mu(A')\mu(E_n)\right|\\
&=\frac1{\mu(A)}\left|\mu(A'\cap E_n)-\mu(A')\mu(E_n)\right|
\end{align*}
since every point in $A\sm A'$ has $\tau_A=p\le n$.  Given $j\le n$, we make some approximations: 
$$\left|\mu(A'\cap E_n)-\mu(A'\cap T^{-j}E_{n-j})\right|\le \mu(A'\cap\{\tau_A\le j\}),$$
$$\left|\mu(A'\cap T^{-j}E_{n-j})-\mu(E_{n-j})\mu(A')\right|\le C\|1_{A'}\|_{\C_2}\mu(E_{n-j})j^{-\beta }<CK j^{-\beta},$$
and
$$\left|\mu(T^{-j}E_{n-j})-\mu(E_n)\right|\le \mu(\tau_A\le j).$$
Then
\begin{align*}
\left|\mu_A(E_n)-\frac{\mu(A')}{\mu(A)}\mu(E_n)\right|& \le \mu_A(A'\cap\{\tau_A\le j\})+\frac{CK}{\mu(A)j^{\beta}}+ \frac{\mu(A')}{\mu(A)}\mu(\tau_A\le j).
\end{align*}

This estimate is still satisfied when $j> n$, indeed
\begin{eqnarray*}
\left|\mu_A(E_n)-\frac{\mu(A')}{\mu(A)}\mu(E_n)\right|&\leq& \left|\frac{\mu(A')}{\mu(A)}-\mu_A(E_n)\right|+\left|\frac{\mu(A')}{\mu(A)}-\frac{\mu(A')}{\mu(A)}\mu(E_n)\right| \\
&\leq &\mu_A(A'\cap\{\tau_A< n\})+ \frac{\mu(A')}{\mu(A)}\mu(\tau_A< n)\\
&\leq & \mu_A(A'\cap\{\tau_A\le j\})+ \frac{\mu(A')}{\mu(A)}\mu(\tau_A\le j).
\end{eqnarray*}

Thus, for any $j\in \N$, we get
\begin{equation}\label{eqdeltaaa'}
\delta(A,A') \le \mu_A(A'\cap\{\tau_A\le j\})+\frac{CK}{\mu(A)j^{\beta}}+ \frac{\mu(A')}{\mu(A)}\mu(\tau_A\le j).
\end{equation}

Let us choose $j=\left\lceil\frac{1}{\mu(A)^\gamma}\right\rceil$ where $1>\gamma>\frac{1}{\beta}$. Then $j\mu(A)$ and $1/(j^\beta\mu(A))$ are both small in $\mu(A)$, so in particular the second term goes to zero as $r$ goes to zero. 

For the first term, since $z$ is a repelling periodic point, there exists $\alpha>0$ so that $f^p$ is behaving locally, in the domain of repulsion, like its linearisation $y\mapsto e^{\alpha p}y$.  Hence if $x\in A'$, then $x$ must leave the domain of repulsion of $z$ before it can return to $A$, giving $\tau_A> \alpha \log (1/r)$.  
Therefore,
\begin{align*}
\mu_A(A'\cap\{\tau_A\le j\}) &\le \frac1{\mu(A)}\sum_{k=-\alpha \log r}^j \mu(A'\cap f^{-k}A)\\
&\le \mu(A')(j+\alpha\log r)+\frac{C\|1_{A}\|_{\C} \mu(A')}{\mu(A)}\sum_{k=\lfloor-\alpha \log r\rfloor}^j k^{-\beta}\\
&\le j\mu(A')+CK\frac{1}{\lfloor\alpha \log(1/ r)\rfloor^{\beta+1}}.
\end{align*}

Clearly the second term decays to zero. For the first term we note that $\mu(A')\le \mu(A)$ and we chose $j$ so that $j\mu(A)$  decays to 0 with $r$.

Finally we note that, in \eqref{eqdeltaaa'}, 
$\mu(\tau_A\le j)  \le j\mu(A)$
which decays to zero as above, so the lemma is proved.
\end{proof}

\begin{proof}[Proof of Theorem~\ref{thm:main det}]
Let $t>0$ and $n=\left\lfloor\frac{t}{\mu(A)}\right\rfloor$. Using Lemma~\ref{lem:exp approx},
\begin{eqnarray*}
\left|\mu(\tau_A>n)-e^{-\Theta t}\right|&\leq& \frac{\mu(A')}{\mu(A)}\delta(A, A')+\left|(1-\mu(A'))^{n}-e^{-\Theta t}\right|\\
& &+(1-\mu(A'))^{n-p}\left|\mu(\tau_A>p)-(1-\mu(A'))^{p}\right|.
\end{eqnarray*}
Since, as $r\rightarrow 0$, we have $\mu(A')/\mu(A)\rightarrow1/\Theta$, $\mu(\tau_A>p)\rightarrow 1$ and $(1-\mu(A'))^{p}\rightarrow 1$, and since one can prove that $(1-\mu(A'))^{n}$ converges to $e^{-\Theta t}$, the theorem is proved using Lemma~\ref{lem:delta0}.
\end{proof} 

We close this section with a discussion of the interpretation of $\Theta$.  Note that in the deterministic SFT case, with finite alphabet and a Gibbs measure, this is dealt with in \cite[Lemma 6.1]{Hir}.  Let $\phi:X\to \R$ be a sufficiently regular (e.g.\ H\"older) potential with $P(\phi)=0$, such that $\mu$ is an equilibrium state and $m$ is the corresponding $\phi$-conformal measure.  We use the notation $S_n\phi(x)=\phi(x)+\phi\circ\sigma(x)+\ldots+\phi\circ\sigma^{n-1}(x)$.

\begin{lemma}
If $\frac{d\mu}{dm}\in (0,\infty)$ then 
$\lim_{r\to 0}\frac{\mu(B_r(z)\sm f^{-p}B_r(z))}{\mu(B_r(z))}=\Theta$ where $\Theta=1-e^{S_p\phi(z)}$.
\label{lem:ann scale}
\end{lemma}

\begin{proof}
This follows as in \cite[Lemma 3.1]{FreFreTod12} where the lemma is shown for the conformal measure $m$, then adding in the invariant measure via bounded density.
\end{proof}

\section{Proof of the quenched law}\label{sec-proof}

From here on we will be in the random SFT setting described in Section~\ref{sec-result}.  
Our proof of Theorem~\ref{thm:main1} follows the line of \cite{Sau09, RouSauVar13}.

\subsection{Quenched law for cylinders}

Since $\theta$ is invertible, by $\sigma$-invariance of $\nu$ and almost everywhere uniqueness of the 
decomposition $d\nu=d\mu_\omega \, d\mathbb{P} $, the set $$\Omega' := \{\omega\in\Omega\colon \forall i,\ (\sigma^i)_*\mu_\omega=\mu_{\theta^i\omega}\}$$ has full $\mathbb P$-probability.

Given $z\in X$, sets $z\in A'\subset A\subset X$ and $\omega\in \Omega$, consider
\[
\delta_\omega(A, A') = \delta_{z, \omega}(A, A')
	:=\sup_{j\ge p} \left|\mu_\omega(A')\mu_\omega(\tau_A(\cdot)>j) - \mu_\omega(A\cap \{\tau_A(\cdot)>j\})\right|
\]
where $p$ is the period of $z$ if $z$ is a periodic point, and $p=0$ otherwise.

\begin{lemma} \label{lem:delta}
Let $z\in X$.  If $z$ is periodic, set $p$ to be the period; otherwise set $p=0$.  For all $\omega\in\Omega'$, integers $k>p$ and measurable $A'\subset A\subset X$,
\begin{align*}
\left|\mu_\omega(\tau_A>k)-\prod_{i=1}^{k-p}\left(1-\mu_{\theta^i\omega}(A')\right)\mu_{\theta^{k-p}\omega}(\tau_A>p)\right|
\\
&\hspace{-3cm}\le
\sum_{i=1}^{k-p}\delta_{\theta^i\omega}(A,A')\prod_{j=1}^{i-1} \left(1-\mu_{\theta^j\omega}(A')\right).
\end{align*}
\end{lemma}

\begin{proof}
For any integer $i\ge 1$ we have
\[
\begin{split}
\mu_\omega(\tau_A>i+1)
&=
\mu_\omega(  \sigma^{-1}(A^c\cap \{\tau_A>i\})) \\
&=
\mu_{\theta\omega}(\tau_A>i)-\mu_{\theta\omega}(A\cap\{\tau_A>i)\}).
\end{split}
\]
Therefore
$$
\left|
\mu_\omega(\tau_A>i+1) - \left(1-\mu_{\theta\omega}(A')\right)\mu_{\theta\omega}(\tau_A>i) 
\right|
\le
\delta_{\theta\omega}(A, A').
$$
An immediate recursive substitution argument using $i\ge p$ finishes the proof of the lemma.
\end{proof}

The proof of Theorem~\ref{thm:main1} is based on the previous lemma. 
The strategy is to prove that the term $\prod_{i=1}^{k-p}\left(1-\mu_{\theta^i\omega}(A')\right)\mu_{\theta^{k-p}\omega}(\tau_A>p)$  is almost 
surely convergent to $e^{-\Theta t}$, and that the error term in the right hand side goes to zero almost surely.

Let then $t>0$ be fixed. Given $A\subset X$,  let  $k=k_{A,t}=\lfloor t/\mu(A)\rfloor$ 
and define 
\[
M_{A, A',t}(\omega):=\sum_{i=1}^{k_{A,t}}\mu_{\theta^i\omega}(A').
\]

\begin{lemma}\label{cor:conv}
For $z\in A_n'\subset A_n$ with $A_n\to \{z\}$ as $n\to \infty$ and satisfying \eqref{measurecyl}, as $A_n\to \{z\}$,
\[
\prod_{i=1}^{k_{A_n,t}-p}\left(1-\mu_{\theta^i\omega}(A_n')\right)\mu_{\theta^{k_{A_n,t}-p}\omega}(\tau_{A_n}>p) - e^{-M_{A_n, A_n',t}(\omega)} \to 0.
\]

\end{lemma}
From now on, to reduce notation, we remove the dependence in $n$ and denote $A=A_n$ and $A'= A_n'$.
\begin{proof}
By \eqref{measurecyl}, as $A\to \{z\}$, 
$\sup_\omega\mu_\omega(A)\to0$, so
\[
\prod_{i=1}^{k_{A,t}}\left(1-\mu_{\theta^i\omega}(A')\right)  - e^{-M_{A, A',t}(\omega)} \to 0,
\]
which is a consequence of the following simple result: 
if $0<\eps\le 1/2$ and $x_1,\ldots,x_k\in[0,\eps]$ then 
\[
\exp\left(-(1+2\eps)\sum_{i=1}^k x_i\right) \le \prod_{i=1}^k(1-x_i) \le \exp\left(-(1-2\eps)\sum_{i=1}^k x_i\right). 
\]
The lemma then follows if we can show that as $A\to \{z\}$, 
\[\prod_{i=1}^{k_{A,t}-p}\left(1-\mu_{\theta^i\omega}(A')\right)\mu_{\theta^{k_{A,t}-p}\omega}(\tau_A>p)-\prod_{i=1}^{k_{A,t}}\left(1-\mu_{\theta^i\omega}(A')\right)\to 0.\]
First note that
\begin{align}
 &\left|\prod_{i=1}^{k_{A,t}-p}\left(1-\mu_{\theta^i\omega}(A')\right)\mu_{\theta^{k_{A,t}-p}\omega}(\tau_A>p)-\prod_{i=1}^{k_{A,t}}\left(1-\mu_{\theta^i\omega}(A')\right)\right|\nonumber\\
&\hspace{2.5cm}\leq \left|\mu_{\theta^{k_{A,t}-p}\omega}(\tau_A>p)-\prod_{i=k_{A,t}-p+1}^{k_{A,t}}\left(1-\mu_{\theta^i\omega}(A')\right)\right|.\label{eq:conv}
\end{align}
Now
\[\mu_{\theta^{k_{A,t}-p}\omega}(\tau_A>p)\to 1\]
since 
\begin{eqnarray*}
\mu_{\theta^{k_{A,t}-p}\omega}(\tau_A>p)&=& 1-\mu_{\theta^{k_{A,t}-p}\omega}(\tau_A\leq p)\\
&\geq& 1-\sum_{i=1}^{p}\mu_{\theta^{k_{A,t}-p}\omega}(\sigma^{-i}A) \\
&\geq& 1-pc_1 e^{-h_1 n}
\end{eqnarray*}
by \eqref{measurecyl}.
Moreover, again by \eqref{measurecyl},
\[\prod_{i=k_{A,t}-p+1}^{k_{A,t}}\left(1-\mu_{\theta^i\omega}(A')\right) \to 1\]
since
\[\prod_{i=k_{A,t}-p+1}^{k_{A,t}}\left(1-\mu_{\theta^i\omega}(A')\right)\geq \left(1-c_1 e^{-h_1 n}\right)^p.\]
Thus the RHS of \eqref{eq:conv} converges to 0 as $A\to \{z\}$, as required.
\end{proof}

Observe that by stationarity the expectation of $M_{A, A',t}$ is
\begin{equation}\label{eq:EMAt}
\EE(M_{A, A',t})=\int_\Omega \sum_{i=1}^{k_{A,t}}\; \mu_{\theta^i\omega}(A') \; d\PP(\omega) = k_{A,t}\mu(A'),
\end{equation}
which, under the assumption that $\frac{\mu(A')}{\mu(A)}\to \Theta$ means that by definition of $k_{A,t}$ we obtain $\EE(M_{A, A',t})\to \Theta t$ as $\mu(A)\to0$.

We can summarise by saying that to prove our limiting law, we are led to prove that $M_{A, A',t}\to \Theta t$ and
the error term $\sum_{i=1}^{k-p}\delta_{\theta^i\omega}(A,A')$ goes to zero  $\PP$-a.e. as $A$ shrinks to $z$.  

Up to this point in this section, we have left our choice of sets $A, A'$ to be rather flexible: since we have only applied \eqref{measurecyl}, the proofs above apply well beyond the symbolic setting.  From now on, for a point $z\in X$, we will take $A$ to be the $n$-cylinder $C_n(z)$ and
\begin{equation*}
A' =\begin{cases} A= C_n(z) &\text{ if } z \text{ is not periodic},\\
C_n(z) \sm C_{n+p}(z) & \text{ if } z \text{ is }p\text{-periodic}.
\end{cases}
\end{equation*}
Note that in the latter case, $A'$ will consist of a collection of $(n+p)$-cylinders.  So $\Theta$ is the limit of $\frac{\mu(A')}{\mu(A)}$, which, in the non-periodic case trivially exists and equals 1.

\begin{remark}
In fact what we are interested in is $A'=C_n(z)\sm \sigma^{-p}(C_n(z))$, which for $n$ large enough is precisely what we have in both cases above: since in the $p$-periodic case $C_n(z)\cap \sigma^{-p}(C_n(z))=C_{n+p}(z)$; while in the non-periodic case, for each finite $p$, for $n$ large enough $C_n(z)\cap \sigma^{-p}(C_n(z))=\es$.  This is one way to see intuitively why the dichotomy should be true.  Moreover, if $z$ is a point of arbitrarily large period (`less and less periodic'), one can see that $\Theta$ becomes arbitrarily close to 1.  
\end{remark}

\begin{remark}
In our main proofs we assume (III), but in fact we apply it using $A'$ in place of $A$ and $A$ in place of $B$.  Recall that in the periodic case $A'$ can be expressed as a union of $(n+p)$-cylinders, so (III) implies
\begin{align*}
\left|\mu_\omega( A'\cap \sigma^{-k-n} A)-\mu_\omega(A')\mu_{\theta^{n+k}\omega}(A)\right|&=\sum_{\hat A\in \F_{n+p}\cap A'} \left|\mu_\omega( \hat A\cap \sigma^{-k-n} A)-\mu_\omega(\hat A)\mu_{\theta^{n+k}\omega}(A)\right|\\
&\hspace{-3cm}=\sum_{\hat A\in \F_{n+p}\cap A'} \left|\mu_\omega( \hat A\cap \sigma^{-(k-p)-(n+p)} A)-\mu_\omega(\hat A)\mu_{\theta^{n+k}\omega}(A)\right|\\
&\hspace{-3cm} \le  \psi(k-p)\mu_\omega(A')\mu_{\theta^{n+k}\omega}(A),
\end{align*}
which is what we use in our proofs.
\end{remark}

Let us take $A=C_n(z)$.
For simplicity we denote $M_{A, A',t}$ by $M_n$ and $k_{A,t}$ by $k_n$. 

\begin{lemma}\label{lem:MN}
For all $z \in X$ we have $\limsup_n M_n\le t$ $\PP$-almost surely, and if the limit $\Theta$ exists, then $M_n\to \Theta t$, $\PP$-almost surely.
\end{lemma}

\begin{proof}
We estimate the second moment of $M_n$
$$
\EE(M_n^2) 
	= \sum_{i,j=1}^{k_n} \int_{\Omega}\mu_{\theta^i\omega}(A')\mu_{\theta^j\omega}(A')d\PP(\omega).
$$

Let $\eps>0$ and consider now $m=m_n=\lfloor e^{h_1n/(1+\eps)}\rfloor$. Near the diagonal, that is when $|i-j|<m$, using \eqref{measurecyl} we have that
\[
\begin{split}
\sum_{|i-j|<m} \int_{\Omega} \mu_{\theta^i\omega}(A')\mu_{\theta^j\omega}(A')d\PP(\omega)
&\leq
\sum_{|i-j|<m} \int_{\Omega} \mu_{\theta^i\omega}(A)\mu_{\theta^j\omega}(A)d\PP(\omega)\\
&\le
\sum_{|i-j|<m} c_1e^{-h_1n} \int_\Omega \mu_{\theta^i\omega}(A)d\PP(\omega)\\
&\le
2c_1 m e^{-h_1n} k_n\mu(A)\\
&\le
2c_1 t m e^{-h_1n} .
\end{split}
\]
In the following estimates, we set $p=0$ when $z$ is not a periodic point.
 Far from the diagonal, the independence hypotheses (I) and (III) yield
\begin{equation}\label{eqfardiagmn} 
\begin{split}
\sum_{|i-j|\ge m}^{k_n}\int_{\Omega} \mu_{\theta^i\omega}(A')\mu_{\theta^j\omega}(A')d\PP(\omega)
&\le
2\sum_{j\ge i+m}^{k_n}\int_{\Omega}  \psi(m-n-p)\mu_{\theta^i\omega}(A')\mu_{\theta^j\omega}(A')d\PP(\omega) +\\
&\quad\quad\quad\quad\quad + \int_\Omega\mu_{\theta^i\omega}(A'\cap\sigma^{-(j-i)}A')d\PP(\omega)\\
&\le 
2  k_n \psi(m-n-p) + 2 \sum_{j\ge i+m}^{k_n} \mu(A'\cap\sigma^{-(j-i)}A') \\
&\le
2   k_n\psi(m-n-p) + k_n^2\mu(A')^2 + k_n^2   \psi(m-n-p).
\end{split}
\end{equation}

Combining these estimates with \eqref{eq:EMAt}, which gives $\EE(M_n)=k_n\mu(A')$, 
we finally get a control on the variance of $M_n$ 
\begin{eqnarray}
\var M_n
&=& \EE(M_n^2) -\EE(M_n)^2\nonumber\\
& \le& 
2c_1 t m e^{-h_1n} +
2   k_n\psi(m-n-p) +   k_n^2 \psi(m-n-p). \label{eq:varmn}
\end{eqnarray}

Thus, one can choose $\eps$ small enough such that $\sum_n \var M_n<\infty$. Indeed, for $n$ large
enough $k_n\psi(m-n-p)< k_n^2 \psi(m-n-p)<c_0^2 2^{q}t^{2}e^{n\gamma}$ where
$\gamma=2h_0-q\frac{h_1}{1+\eps}$ satisfies $\gamma<0$ if $\eps$ is sufficiently small, by our assumption on $q$. 
Since any sequence of centred random variables  $(X_n)$ with
$\sum_n \var X_n<\infty$ is such that $X_n\to0$ a.s.,
we have $M_n-\EE(M_n)\to0$ a.s., from which the conclusion follows since in all cases $\limsup_n \EE(M_n)\le t$, and when $\Theta$ exists then $\EE(M_n)\to \Theta t$.
\end{proof}

Next, the error term $\sum_{i=1}^{k-p}\delta_{\theta^i\omega}(A,A')$ in Lemma~\ref{lem:delta} decomposes 
as a mixing term and short entrance or return time terms as follows. 
Let $g\le k$ be an integer, 
and set
\[
\begin{split} 
G_{A,k,g}(\omega) &= \sum_{i=1}^k \mu_{\theta^i\omega}(A'\cap \{\tau_A\le g\}),\\
H_{A,k,g}(\omega) &= \sum_{i=1}^k \sup_{j\ge p} \left|\mu_{\theta^i\omega}(A'\cap \sigma^{-g}\{\tau_A>j\})-\mu_{\theta^i\omega}(A')\mu_{\theta^{i+g}\omega}(\tau_A>j)\right|,\\
K_{A,k,g}(\omega) &= \sum_{i=1}^k \mu_{\theta^i\omega}(A') \, \mu_{\theta^i\omega}(\tau_A\le g).
\end{split}
\]

The gap $g$ allows to exploit the mixing assumptions, related to $H_{A,k,g}$, provided that the 
probabilities of hitting or returning into $A$ before time $g$, related to $G_{A,k,g}$ and $K_{A,k,g}$, 
are small since the whole error term is estimated as follows.

\begin{lemma}\label{lem:somme}
For all $\omega\in\Omega'$, any measurable set $A\subset X$ and any integers $g\le k$ we have 
$\displaystyle \sum_{i=1}^{k-p}\delta_{\theta^i\omega}(A,A') \le  G_{A,k,g}(\omega)+H_{A,k,g}(\omega)+K_{A,k,g}(\omega)$.
\end{lemma}

\begin{proof} We have
\begin{eqnarray*}
\delta_\omega(A, A') 
&=& \sup_{j\ge p} |\mu_\omega(\tau_A>j)\mu_\omega(A')-\mu_\omega(A\cap\{\tau_A>j\})|\\
&=& \sup_{j\ge p} |\mu_\omega(\tau_A>j)\mu_\omega(A')-\mu_\omega(A'\cap\{\tau_A>j\})\\
& &+ \mu_\omega(A'\cap\{\tau_A>j\})-\mu_\omega(A\cap\{\tau_A>j\})|\\
&\le&
\mu_\omega(\tau_A\le g) \, \mu_\omega(A') +\mu_\omega(A'\cap \{\tau_A\le g\}) \\
& & +\sup_{j\ge g} \left| \mu_{\theta^g\omega}\left(\tau_A>j-g\right) \, \mu_\omega(A')
- \mu_\omega\left(A'\cap \sigma^{-g}\{\tau_A>j-g\}\right) \right|
\end{eqnarray*}
since for $j\ge p$ we have $ \mu_\omega(A'\cap\{\tau_A>j\})=\mu_\omega(A\cap\{\tau_A>j\})$.
Thus the lemma follows by summing up the the previous terms along the finite piece of orbit of $\omega$ by $\theta$. 
\end{proof}

We forget the dependence on $z$, $g$ and $t$ for these random variables and hence we write $G_n$, $H_n$, $K_n$ for notational simplicity. We now prove that they converge to
zero as $n$ tends to infinity. 
We fix a gap of size $g=g_n= e^{h_1n/2}$.

\begin{remark}
Note that in the remaining lemmas in this section, when we appeal to Lemma~\ref{lem:MN}, we actually only require that $M_n$ is uniformly bounded above.
\end{remark}

\begin{lemma} \label{lem:HNas}
For all $z\in X$ we have $H_n\to0$ $\PP$-almost surely.
\end{lemma}

\begin{proof}
We use the correlation hypothesis (III) to obtain
\[
H_n(\omega)\le  \sum_{i=1}^{k_n}  \psi(g-n)\mu_{\theta^i\omega}(A')\leq  \psi(g-n)M_n.
\]
Thus Lemma~\ref{lem:MN} and the definition of $\psi$ give us that $H_n \to 0$ $\PP$-almost surely.
\end{proof}

\begin{lemma} \label{lem:KNas}
For all $z\in X$ we have $K_n\to0$, $\PP$-almost surely.
\end{lemma}
\begin{proof} We have $K_n\le gc e^{-h_1n} M_n$ since 
\[
\mu_\omega( \tau_A(\cdot)\le g)
	=\mu_\omega\Big(\bigcup_{i=1}^g\sigma^{-i}A\Big)
	\le \sum_{i=1}^g \mu_{\theta^i\omega}(A)\le gc_1e^{-h_1n}.
\] 
In addition, $M_n\le t$ $\PP$-almost surely by Lemma~\ref{lem:MN}, which gives the conclusion.
\end{proof}

Proving that $G_n\to 0$ a.s.\ is the only time we really use the periodicity, or non-periodicity of $z$.  We require some preparatory results.  We write
\begin{eqnarray*}
G_n &=& \sum_{i=1}^{k_n} \mu_{\theta^i\omega}(A'\cap \{ \tau_A(\cdot)\le g\})\\
&\leq &\sum_{i=1}^{k_n} \sum_{j=1}^{\vert A \vert}\mu_{\theta^i\omega}(A'\cap\sigma^{-j}A)+ \sum_{i=1}^{k_n} \sum_{j=\vert A \vert+1}^g\mu_{\theta^i\omega}(A'\cap\sigma^{-j}A),
\end{eqnarray*}
where $|A|$ denotes the depth of the cylinder, i.e., $|A|=n$ here. 
In the case that $z$ is a periodic point then a point in $A'$ cannot return to $A$ before time $|A|$, so the first summand is null in this case.  For the non-periodic case, where $A'=A$, we require two elementary lemmas (see also \cite[Section 6]{FreFreTod12}).

\begin{lemma}
Given a non-periodic point $z$, let $A_n$ be the sequence of $n$-cylinders around $z$.  Then there exists a sequence $p_n\to \infty$ as $n\to \infty$ such that the first return of any point in $A_n$ to $A_n$ is at least $p_n$.  Moreover, for each $j\le n$, there is at most one $(n+j)$-cylinder $B\subset A_n$ which returns to $A_n$ at time $j$.
\label{lem:sparse early returns}
\end{lemma}

\begin{proof}
Let $z=z_0z_1\ldots$.  If such $(p_n)_n$ did not exist then there would be some $p$ such that $z=z_0\ldots z_{p-1}z_0\ldots z_{p-1}z_0\ldots$ and $z$ would be periodic.  Moreover, if $y\in A_n=[z_0\ldots z_{n-1}]$ returns to $A_n$ at time $j<n$, then we must have $z_j=z_0, z_{j+1}=z_1, \ldots z_{n-1}=z_{n-1-j}$ and $y$ must lie in the $(n+j)$-cylinder $[z_0\ldots z_{n-1}z_{n-j}z_{n-j+1}\ldots z_{n-1}]$.  Since $y$ was any point returning at time $j<n$, the second part of the lemma is proved.
\end{proof}

We are now in a position to prove the following lemma.

\begin{lemma}
If $z$ is a non-periodic point and $A$ represents an $n$-cylinder around $z$, then 
$$\sum_{i=1}^{k_n} \sum_{j=1}^{\vert A \vert}\mu_{\theta^i\omega}(A\cap\sigma^{-j}A) \to 0$$
as $n\to \infty$.
\label{lem:v short rets}
\end{lemma}

\begin{proof}
By Lemma~\ref{lem:sparse early returns} and \eqref{measuresumcyl}, 
\begin{align*}
\sum_{i=1}^{k_n} \sum_{j=1}^{\vert A \vert}\mu_{\theta^i\omega}(A\cap\sigma^{-j}A)&=\sum_{i=1}^{k_n} \sum_{j=p_n}^{\vert A \vert}\mu_{\theta^i\omega}(A\cap\sigma^{-j}A)&\le \sum_{i=1}^{k_n} c_2e^{-h_1 p_n}\mu_{\theta^i\omega}(A) \\
&\le  Ce^{-h_1 p_n}\end{align*}
where the final inequality follows from Lemma~\ref{lem:MN}. 
Since $p_n\to \infty$ as $n\to \infty$, the lemma holds.
\end{proof}

Finally we  prove $G_n\to 0$ a.s.

\begin{lemma} \label{lem:GNas}
For every $z \in X$ we have $ G_n\to 0$ $\PP$-almost surely. 
\end{lemma}
\begin{proof}
We assume here that $z$ is $p$-periodic since the complementary case follows immediately from this argument (with $A'=A$) and Lemma~\ref{lem:v short rets}.
Using the hypothesis on the decay of correlations, 
\begin{eqnarray*}
G_n &=& \sum_{i=1}^{k_n} \mu_{\theta^i\omega}(A'\cap \{ \tau_A(\cdot)\le g\})\\
&\leq &\sum_{i=1}^{k_n} \sum_{j=\vert A \vert}^g\mu_{\theta^i\omega}(A'\cap\sigma^{-j}A)\\
&\leq& \sum_{i=1}^{k_n}\left(\sum_{j=\vert A \vert}^g \mu_{\theta^i\omega}(A')\mu_{\theta^{i+j}\omega}(A)+ \psi(j-\vert A\vert)\mu_{\theta^i\omega}(A')\mu_{\theta^{i+j}\omega}(A)\right)\\
&\leq& c_1(1+ \psi(0))g e^{-h_1 n}\sum_{i=1}^{k_n}\mu_{\theta^i\omega}(A')
\end{eqnarray*}
Since by Lemma~\ref{lem:MN} $\limsup_n\sum_{i=1}^{k_n}\mu_{\theta^i\omega}(A')\le t$, $\PP$-almost surely, the lemma is proved. 
\end{proof}

Combining the lemmas in this section completes the proof of Theorem~\ref{thm:main1}.  Corollary~\ref{cor:main1} follows immediately by integrating over $\omega$.

\subsection{Infinite alphabets}
This subsection is devoted to the proof of Theorem~\ref{theoinfalph}.

The difference between Theorem~\ref{thprin} and Theorem~\ref{theoinfalph} is that we do not assume (II) and thus allow infinite alphabet: this means that we lose some control of $k_n$, which we compensate for in the strengthened mixing assumption in (I') compared to (I). In the proof of Theorem~\ref{thprin}, assumption (II) is only used in the proof of Lemma~\ref{lem:MN} and in particular to prove that $\sum_n\var(M_n)<\infty$. Under the assumptions of Theorem~\ref{theoinfalph}, we replace \eqref{eqfardiagmn} with:

\begin{align*} 
\sum_{|i-j|\ge m}^{k_n}\int_{\Omega} \mu_{\theta^i\omega}(A')\mu_{\theta^j\omega}(A')d\PP(\omega)
&\le
2\sum_{j\ge i+m}^{k_n}\int_{\Omega}  \psi(m-n-p)\mu_{\theta^i\omega}(A')\mu_{\theta^j\omega}(A')d\PP(\omega) \\
&\quad + \int_\Omega\mu_{\theta^i\omega}(A'\cap\sigma^{-(j-i)}A')d\PP(\omega)\\
&\hspace{-20mm}\le 
2  k_n \psi(m-n-p)\mu(A') + 2 \sum_{j\ge i+m}^{k_n} \mu(A'\cap\sigma^{-(j-i)}A') \\
&\hspace{-20mm}\le
2   k_n\psi(m-n-p)\mu(A') + k_n^2\mu(A')^2 (1+   \psi(m-n-p)).
\end{align*} 
 Similarly to \eqref{eq:varmn}, we obtain:
\[
\var M_n \le 
2c_1 t m e^{-h_1n} +
2   k_n\psi(m-n-p)\mu(A') +   k_n^2 \psi(m-n-p)\mu(A')^2. 
\]
Thus, $\sum_n\var(M_n)<\infty$ since $k_n\mu(A')\leq t$ and the conclusions of Lemma~\ref{lem:MN} are satisfied. This completes the proof of Theorem~\ref{theoinfalph}.
\section{Existence of the extremal index}
\label{sec:Theta ex}

Here we address the issue of the existence of the limit $\Theta=\lim_{n\to \infty}\frac{\mu\left(C_n(z)\sm C_{n+p}(z)\right)}{\mu(C_n(z))}$, which follows from the existence of $\lim_{n\to \infty}\frac{\mu\left(C_{n+p}(z)\right)}{\mu(C_n(z))}$.
Notice that we are considering the measure $\mu$, rather than the sample measures $\mu_\omega$ here.  However, since we integrate over the $\mu_\omega$ to get $\mu$, we need only consider the general properties of the sample measures, while for ease of notation we will suppress this difference.

For a dynamical system $f:X\to X$, a measure $m$ on $X$ is called $\phi$-conformal for some observable $\phi:X\to \R$, if whenever a measurable set $A$ is such that $f:A\to f(A)$ is a bijection, then
$$m(f(A))=\int_Ae^{-\phi}~dm.$$

In the cases we will deal with below we will have a continuous function $\phi$ along with a $\phi$-conformal measure $\nu$ and an invariant measure $\mu\ll \nu$.  If, as in our later examples, we also have a finite density $\frac{d\mu}{d\nu}$ at $z$, then 
$$\lim_{n\to \infty}\frac{\mu\left(C_{n+p}(z)\right)}{\mu(C_n(z))}= \lim_{n\to \infty}\frac{\nu\left(C_{n+p}(z)\right)}{\nu(C_n(z))}.
$$
But since our systems are Markov, we have that $f^p:C_{n+p}(z)\to C_n(z)$ is a bijection, so by conformality of $\nu$ 
$$\nu(C_n)=\int_{C_{n+p}}e^{-S_p\phi}~d\nu$$
which by continuity of $\phi$ converges to $\nu(C_{n+p})e^{-S_p\phi(z)}$ as $n\to \infty$.
Incorporating randomness into this calculation and then integrating yields the same result, so the limit we require does indeed exist.

\section{Random Gibbs measures}
\label{sec:examples}

In this section we will give details of a family of shifts which satisfy our assumptions. 
The simplest non-trivial examples of random dynamical systems to which our results apply are given in \cite[Example 19]{RouSauVar13}: the dynamics on the base $\theta:\Omega\to \Omega$ is a subshift of finite type (SFT) equipped with an invariant probability measure which is a Gibbs measure for a H\"older potential, while on the fibres we have full shifts with Bernoulli measures.  Here we generalise this setting.  To prove (III) we can deal with fairly general shifts on the fibres and we will describe these shifts in some detail.  However, in order to ensure that (I) holds, it is necessary to restrict our combination of base and fibre transformations somewhat. Note that the existence of the limit $\Theta$ follows as in the previous section.

\subsection{The fibre maps and condition (III)}

To guarantee (III) the principal example here is a random shift on an at most countable alphabet, and the corresponding random Gibbs measures.  For generality, we will use the approach detailed in \cite{Sta13} which is concerned with shifts on $\N$, for example the full shift.  We note that this extends a little beyond the full shift, to the so-called BIP setting.  To obtain condition (I), it is most convenient to restrict ourselves to the finite shift case, but we note that this also extends to the countable case if $(\Omega,\p, \theta)$ is sufficiently simple.

As above, we assume that $(\Omega, \p, \theta)$ is an invertible measure preserving system and let $X=\N^{\N_0}$ and let $\sigma:X\to X$ denote the shift (so far we won't assume anything on the structure of the shift spaces).  For $r\in (0,1)$, let $d_r$ be the usual symbolic metric on $X$, i.e., $d_r(x, y)=r^k$ where $x_i=y_i$ for $i=0, \ldots, k-1$, but $x_k\neq y_k$.

Assume that $\phi:X\times \Omega:\to \R$ is a function which is almost surely H\"older continuous, which is to say, for $$V_n^\omega(\phi):=\sup\{|\phi_\omega(x)-\phi_\omega(y)|:x_i=y_i,\ i=0,\ldots, n-1\},$$
there is some $r\in (0,1)$ and $\kappa(\omega)\ge 0$ such that $\int\log\kappa~d\p<\infty$ where $V_n^\omega(\phi)\le \kappa(\omega)r^n$.

Define $S_n\phi_\omega(x):=\sum_{k=0}^{n-1}\phi_{\theta^{k}\omega}\circ\sigma^k(x)$.  If $x,y$ are in the same $m$-cylinder for $m\ge n$, then $|S_n\phi_\omega(x)-S_n\phi_\omega(y)|\le r^{m-n}\sum_{k=0}^{n-1}r^k\kappa(\theta^{n-k}\omega)$.  As in the proof of  \cite[Lemma 7.2]{DenKifSta08-2}, the assumption on the integrability of $\log\kappa$ implies that the above limit is finite a.s., say $\sum_{k=0}^{n-1}r^k\kappa(\theta^{n-k}\omega) \le c_\omega$.  However, it is also pointed out in \cite{Sta13} that if $\kappa$ is integrable, then we have an a.s. uniform upper bound, say $C_\phi$ on $\sum_{k=0}^{n-1}r^k\kappa(\theta^{n-k}\omega)$.  Given a H\"older function $\psi$, then we define
$$D_\omega(\psi):=\sup_{x,y\in X_\omega}\left\{\frac{|\psi(x)-\psi(y)|}{d_r(x, y)}\right\}.$$

Now we define the random Ruelle operator by
$$\L_\omega\psi(x)=\sum_{\sigma y=x}e^{\phi_\omega(y)}\psi_\omega(y)$$
where $\psi:X'\to [0,\infty]$ where $X'\subset X$ is such that $\L_\omega$ is well-defined.  It can be shown that there exists some constant $\lambda_\omega$ which is the maximal eigenvalue for $\L_\omega 1$.  As in \cite{DenKifSta08, Sta13}, we can assume that there exists $\rho_\omega$ which is uniformly bounded from below and such that  $\L_\omega\rho_\omega=\lambda_\omega\rho_{\theta\omega}$  a.s.  and such that $\log\rho$ satisfies the same smoothness properties as $\phi$, i.e. we have the same $\kappa$ and $r$ in the variation.  This allows us to replace $\phi$ with
$$\varphi_\omega(x) :=\phi_\omega(x)+\log \rho_\omega-\log\rho_{\theta\omega}(\sigma x)-\log\lambda_\omega.$$
Letting $\L_\omega$ denote the corresponding transfer operator, one consequence of this is that $\L_\omega 1=1$.  Note also that random equilibrium states for $\phi$ and $\varphi$ coincide.

 Now we have the property that
 \begin{equation}
 \int\L_\omega^n(\psi)\cdot \gamma~d\mu_\omega=\int\psi\cdot \gamma\circ \sigma^n~d\mu_{\theta^{-n}\omega}\label{eq:L basic}
 \end{equation}
for appropriate observables $\psi, \gamma$.

We will make the following almost sure assumptions on our system (which are easily satisfied for SFTs with H\"older potentials):

\begin{enumerate}

\item $\int\kappa~d\p<\infty$, so $\sum_{k=0}^\infty r^k\kappa(\theta^{n-k}\omega)$ is a.s. uniformly bounded, independently of $\omega$.

\item There exists a measure $\mu_\omega$ where $\L_\omega^*\mu_\omega=\mu_{\theta^{-1}\omega}$, i.e., \eqref{eq:L basic} holds for $L^1$ observables. 

\item Big images:  there exists some $C_{BIP}>0$ such that for any $n$-cylinder $U$ and $\omega\in \Omega$, $\inf1/{\mu_{\theta^{n}\omega}(\sigma^{n}U)}>C_{BIP}$.

\item There exist $C>0$, and $g(n)\to 0$ as $n\to \infty$ such that 
$$\left\|\L_\omega^n(\psi)-\int\psi~d\mu_{\omega}\right\|_\infty\le Cg(n)D_\omega(\psi).$$
\end{enumerate}

Given an $n$-cylinder $A\subset X_\omega$, let $\hat\sigma_{\omega, A}:\sigma^nA\to A$ denote the local inverse of $\sigma^n$. 
We use the previous conditions to obtain the decay of correlations for the sample measures:

\begin{proposition}
Under the above conditions there exists $C>0$ such that
if $A$ is an $n$-cylinder and $B$ an $m$-cylinder 
$$\left|\mu_\omega(A\cap \sigma_{\omega}^{-j-n} B)-\mu_\omega(A)\mu_{\theta^{n+j}\omega}(B)\right|\le Cg(j)\mu_\omega(A)\mu_{\theta^{n+j}\omega}(B).$$
\end{proposition}

\begin{proof}
We first relate the problem to the behaviour of the transfer operator as follows.  

By \eqref{eq:L basic}, 
\begin{align*}
\mu_\omega\left(A \cap(\sigma^{-k-n}B)\right)= \int 1_A\cdot 1_B\circ \sigma^{k+n}~d\mu_\omega
=\int \L_{\theta^n\omega}^{k}\left(\L^n1_A\right)\cdot 1_B~d\mu_{\theta^{n+k}\omega}.
\end{align*}

So noticing that since $A$ is an $n$-cylinder 
$$\left(\L^n1_A\right)(x) =\sum_{\sigma^ny=x}e^{S_n\varphi_\omega(y)}1_A(y)=\left((e^{S_n\varphi_\omega})\circ\hat\sigma_{\omega, A}\right)(x),$$
we obtain
\begin{align*}
\mu_\omega\left(A \cap(\sigma^{-k-n}B)\right) =\int\L_{\theta^n\omega}^k\left((e^{S_n\varphi_\omega})\circ\hat\sigma_{\omega, A}\right)\cdot 1_B~d\mu_{\theta^{n+k}\omega}.
\end{align*}

Since similarly 
$$\mu_\omega(A)=\int\L_\omega^n 1_A~d\mu_{\theta^n\omega} =\int \left(e^{S_n\varphi_\omega}\right)\circ\hat\sigma_{\omega, A}~d\mu_{\theta^n\omega}=\int_{\sigma^nA} \left(e^{S_n\varphi_\omega}\right)\circ\hat\sigma_{\omega, A}~d\mu_{\theta^n\omega},$$
putting this into the main estimate here, we obtain by (4) that

\begin{align*}
&\left|\mu_\omega( A\cap \sigma_{\omega}^{-k-n} B)-\mu_\omega( A)\mu_{\theta^{n+k}\omega}(B)\right| \\
& \le\int\left|\L_{\theta^n\omega}^k\left((e^{S_n\varphi_\omega})\circ\hat\sigma_{\omega, A}\right)\cdot 1_B -\mu_\omega(A)\cdot 1_B\right|~d\mu_{\theta^{n+k}\omega}\\
&\le \mu_{\theta^{n+k}\omega}(B)\left\|\L_{\theta^n\omega}^k\left((e^{S_n\varphi_\omega})\circ\hat\sigma_{\omega, A}\right)-\mu_\omega( A)\right\|_\infty\\
&= \mu_{\theta^{n+k}\omega}(B)\left\|\L_{\theta^n\omega}^k\left((e^{S_n\varphi_\omega})\circ\hat\sigma_{\omega, A}\right)-\int \left(e^{S_n\varphi_\omega}\right)\circ\hat\sigma_{\omega, A}~d\mu_{\theta^n\omega}\right\|_\infty\\
&\le \mu_{\theta^{n+k}\omega}(B)Cg(k)D_{\theta^n\omega}\left((e^{S_n\varphi_\omega})\circ\hat\sigma_{\omega, A}\right).
\end{align*}
To estimate $D_{\theta^n\omega}\left((e^{S_n\varphi_\omega})\circ\hat\sigma_{\omega, A}\right)$, if $x, y\in \sigma^n( A)$, then 
\begin{align*}
&\left|\left((e^{S_n\varphi_\omega})\circ\hat\sigma_{\omega, A}\right)(x)-\left((e^{S_n\varphi_\omega})\circ\hat\sigma_{\omega, A}\right)(y)\right|\\
&\hspace{2cm}\le e^{S_n\varphi_\omega\circ\hat\sigma_{\omega, A}(x)}\cdot\left|1-\exp\left(S_n\varphi_\omega \circ\hat\sigma_{\omega, A}(y)- S_n\varphi_\omega \circ\hat\sigma_{\omega, A}(x)\right)\right|\\
&\hspace{2cm}\le \tilde{C}\frac{\mu_{\omega}(A)}{\mu_{\theta^n\omega}(\sigma^nA)}C'C_\varphi d_r(x, y)\\
&\hspace{2cm}\le \tilde{C}C_{BIP}C_\varphi C'\mu_\omega(A)d_r(x, y),
\end{align*}
where the $C'$ is such that $|1-e^x|\le C'|x|$ for small enough $|x|$ and $\tilde{C}=\max\{1,e^{C_\varphi}\}$.  So the proof is complete in the case that $x, y\in \sigma^n(A)$.

In the case that $x\in \sigma^n( A)$ and $y\notin \sigma^n( A)$, we see that 
$$
\left|\left((e^{S_n\varphi_\omega})\circ\hat\sigma_{\omega, A}\right)(x)-\left((e^{S_n\varphi_\omega})\circ\hat\sigma_{\omega, A}\right)(y)\right| = e^{S_n\varphi_\omega\circ\hat\sigma_{\omega, A}(x)},$$
which we can estimate as above.
\end{proof}

\subsection{Decay for the full system}

We next prove (I).  If our fibres are SFTs on a finite alphabet then we can prove the following:

\begin{lemma}  
For any $\alpha\in (0,1]$ there exist $C_\alpha>0$ and $\alpha'\in (0,1]$ such that for an $n$-cylinder $A$, the map $\omega\mapsto \mu_\omega(A)$ is $(C_\alpha r^{-\alpha n}, \alpha')$-H\"older.
 \label{lem:Holder}
\end{lemma}

\begin{proof}
We need to show that the functions $\phi_A( x) = 1_A(x)$ are $(C_A, \alpha_A)$-H\"older.   Then \cite[Theorem 2.10]{DenGor99}  implies that the maps $\phi_A:\omega\mapsto \mu_\omega(A)$ are $(C_AD_{\alpha_A}, \eta(\alpha_A))$-H\"older for $D$ and $\eta$ depending only on $\alpha$.

In fact it is easy to see that $\phi_A$ is actually Lipschitz with Lipschitz constant $r^{-n}$, so for any $\alpha\in (0,1)$, it is $(r^{-n\alpha}, \alpha)$-H\"older.
\end{proof}

To see how this applies in the case where $\theta:\Omega\to \Omega$ is a SFT on a finite alphabet, with a Gibbs measure for a H\"older potential: for $\beta>0$ let the norm $\|\cdot\|_\beta$ be defined by $\|\cdot\|_\beta=|\cdot|_\beta+|\cdot|_\infty$ where $|g|_\beta=\sup\left\{\frac{V_n(g)}{\beta^n}:n\ge 0\right\}$.  So the classical SFT result, see   \cite[Proposition 2.4]{ParPol90} gives
 $$\left|\int v\circ \theta^n\cdot w~d\PP-\int v~d\PP\int w~d\PP\right| \le K\rho^n \|w\|_\beta |v|_\infty,$$
 for $v, w$ H\"older.
  So in our case, letting $\alpha'$ be as in  Lemma~\ref{lem:Holder}, choose $\beta\in (0,\alpha')$.  Then the norm for the characteristic function on the $n$-cylinder $A$ is less or equal to $C_\alpha r^{-\alpha n}$, we obtain
\begin{equation}
\left|\mu(A\cap\sigma^{-g-n}B) -\mu(A)\mu(B)\right|\le K\rho^{n+g} C_\alpha r^{-\alpha n}, \label{eq:base decay}
\end{equation}
 so if we ensure that $\alpha$ is chosen so that $r^\alpha>\rho$, we get an exponential decay like $\rho^g$ and so (I) holds.

Thus we have proved that if the fibre maps satisfy conditions (1)--(4) and the base transformation is an SFT on a finite alphabet with an invariant measure which is a Gibbs state for some H\"older potential, then Theorem~\ref{thm:main1} and Corollary~\ref{cor:main1} apply.  Naturally the base transformation could also be more complicated, so long as it satisfies something like \eqref{eq:base decay}.

\begin{remark}
Notice that if the base transformation is an SFT on a finite alphabet with an invariant measure which is a Gibbs state for some H\"older potential, and the fibre dynamics are countable Markov shifts satisfying (1)--(4) then we can recode the whole dynamical system $(\omega, x)\mapsto (\theta\omega, \sigma x)$ as a (non-random) countable Markov shift which will satisfy (1)--(4) and in fact satisfy (I').
\end{remark}

\textit{Acknowledgments.} The authors would like to thank Nicolai Haydn for several fruitful discussions and comments on a previous version of the paper.  They would also like to thank Manuel Stadlbauer and the referee for useful comments and suggestions.

\end{document}